\documentclass{amsart}
\usepackage{amsthm,amssymb,mathtools}
\usepackage[british]{babel}
\usepackage{enumitem}
\usepackage[latin1]{inputenc}
\usepackage{url}
\usepackage{hyperref}
\usepackage{tikz}
\usetikzlibrary{cd}

\newtheorem{theorem}{Theorem}
\numberwithin{theorem}{section}
\newtheorem{corollary}[theorem]{Corollary}
\newtheorem{lemma}[theorem]{Lemma}
\newtheorem{proposition}[theorem]{Proposition}
{\theoremstyle{definition}
\newtheorem{definition}[theorem]{Definition}
\newtheorem{remark}[theorem]{Remark}

}

\newcommand{\atrs}{\mathbf{ATR_0^{\operatorname{set}}}}
\newcommand{\prs}{\mathbf{PRS\omega}}
\newcommand{\rca}{\mathbf{RCA_0}}
\newcommand{\aca}{\mathbf{ACA_0}}
\newcommand{\atr}{\mathbf{ATR_0}}

\newcommand{\rng}{\operatorname{rng}}
\newcommand{\otyp}{\operatorname{otyp}}

\newcommand{\supp}{\operatorname{supp}}

\newcommand{\en}{\operatorname{en}}

\newcommand{\bh}{\operatorname{BH}}

\newcommand{\lef}{<^{\operatorname{fin}}}

\title{Computable Aspects of the Bachmann-Howard Principle}
\author{Anton Freund}
\address{Anton Freund, Fachbereich Mathematik, Technische Universit\"at Darmstadt, Schloss\-gartenstr.~7, 64289 Darmstadt, Germany}

\begin{document}

\begin{abstract}
We have previously established that $\Pi^1_1$-comprehension is equivalent to the statement that every dilator has a well-founded Bachmann-Howard fixed point, over $\atr$. In the present paper we show that the base theory can be lowered to $\rca$. We also show that the minimal Bachmann-Howard fixed point of a dilator $T$ can be represented by a notation system $\vartheta(T)$, which is computable relative to $T$. The statement that $\vartheta(T)$ is well-founded for any dilator $T$ will still be equivalent to $\Pi^1_1$-comprehension. Thus the latter is split into the computable transformation $T\mapsto\vartheta(T)$ and a statement about the preservation of well-foundedness, over a system of computable mathematics.
\end{abstract}


\maketitle
{\let\thefootnote\relax\footnotetext{This is the submitted version (before peer review) of a paper published in the Journal of Mathematical Logic 20(2) 2020, article no. 2050006, 26 pp, \href{https://doi.org/10.1142/S0219061320500063}{doi:10.1142/S0219061320500063}. Note, in particular, that the numbering of theorems differs from the published version.}

\section{Introduction}

We begin by recalling the abstract Bachmann-Howard principle, which was introduced in \cite{freund-equivalence} (based on the author's PhD thesis~\cite{freund-thesis} and an earlier arXiv preprint~\cite{freund-bh-preprint}). For this purpose we consider the category of linear orders, with order embeddings as morphisms. The forgetful functor to the underlying set of an order will be left implicit. Conversely, we will often view a subset of an ordered set as a suborder. Given a set $X$, we put
\begin{equation*}
 [X]^{<\omega}:=\text{``the set of finite subsets of $X$''}.
\end{equation*}
To obtain a functor we map $f:X\rightarrow Y$ to the function $[f]^{<\omega}:[X]^{<\omega}\rightarrow[Y]^{<\omega}$ with
\begin{equation*}
 [f]^{<\omega}(a):=\{f(s)\,|\,s\in a\}.
\end{equation*}
It is easy to see that $X\mapsto[X]^{<\omega}$ and $f\mapsto[f]^{<\omega}$ are primitive recursive set functions in the sense of Jensen and Karp~\cite{jensen-karp} (with parameter $\omega$). The same will hold for all class functions considered in the sequel. This allows us to formalize our investigation in primitive recursive set theory with infinity ($\prs$), as introduced by Rathjen~\cite{rathjen-set-functions}. Extending $\prs$ by axiom beta and the axiom of countability leads us to $\atrs$, the set-theoretic version of arithmetical transfinite recursion due to Simpson~\cite{simpson82,simpson09}. A detailed introduction to these theories can be found in \cite[Chapter~1]{freund-thesis}. The theory $\prs$ cannot quantify over all primitive recursive set functions. It can, however, quantify over a primitive recursive family of class functions, by quantifying over its set-sized parameters. A definition or proposition which speaks about a collection of class functions (e.g.~about arbitrary endofunctors of linear orders) should be read as a schema: Officially, we have a separate definition or proposition for each primitive recursive family of class functions. Such a family may depend on further sets as parameters (in particular the parameter $\omega$ is often required). We will see that the restriction to primitive recursive set functions does not affect the generality of our results. With these methodological remarks in mind we state the following definition, essentially due to Girard~\cite{girard-pi2}:

\begin{definition}[$\prs$]\label{def:prae-dilator}
 A prae-dilator consists of
 \begin{enumerate}[label=(\roman*)]
  \item an endofunctor $T$ of linear orders and
  \item a natural transformation $\supp^T:T\Rightarrow[\cdot]^{<\omega}$ that computes supports, in the following sense: For any linear order $X$ and any element $\sigma\in T_X$ we have $\sigma\in\rng(T_{\iota_\sigma})$, where $\iota_\sigma:\supp^T_X(\sigma)\hookrightarrow X$ is the inclusion.
 \end{enumerate}
 If $T_X$ is well-founded for any well-order $X$, then $(T,\supp^T)$ is called a dilator.
\end{definition}

Girard's original definition does not include the natural transformation $\supp^T$ but demands that $T$ preserves direct limits and pullbacks. It is straightforward to check that the two formulations are equivalent (see~\cite[Remark~2.2.2]{freund-thesis}), but we find it very helpful to make the support functions explicit. Our prae-dilators are not quite equivalent to Girard's pre-dilators (hence the different spelling), since \cite[Definition~4.4.1]{girard-pi2} contains an additional monotonicity condition. The latter is automatic in the well-founded case, so that it does not make a difference for dilators. According to Girard's definition the values of a dilator have to be ordinals. This has the advantage that isomorphic dilators become equal. Nevertheless we want to allow arbitrary well-orders as values: It will be important that we can represent $T_X\cong\alpha$ by a well-order of rank below $\alpha$ (cf.~Remark~\ref{rmk:dilators-order-types} below). To proceed we introduce the following notation: If $(X,<_X)$ is a linear order (or just a preorder), then the preorder $\lef_X$ on $[X]^{<\omega}$ is defined by
\begin{equation*}
 a\lef_X b\quad:\Leftrightarrow\quad\text{``for any $s\in a$ there is a $t\in b$ with $s<_X t$''.}
\end{equation*}
 In the case of singletons we will write $s\lef_X b$ and $a\lef_X t$ rather than $\{s\}\lef_X b$ resp.~$a\lef_X \{t\}$. The relation $\leq^{\operatorname{fin}}_X$ is defined analogously. In \cite{freund-equivalence} we have introduced the following notion (with precursors in \cite{freund-bh-preprint,freund-thesis}):
 
\begin{definition}[$\prs$]\label{def:bachmann-howard-collapse}
 Consider a prae-dilator $(T,\supp^T)$ and a linear order $X$. A function
 \begin{equation*}
  \vartheta:T_X\rightarrow X
 \end{equation*}
 is called a Bachmann-Howard collapse if the following holds for all $\sigma,\tau\in T_X$:
 \begin{enumerate}[label=(\roman*)]
  \item If we have $\sigma<_{T_X}\tau$ and $\supp^T_X(\sigma)\lef_X\vartheta(\tau)$, then we have $\vartheta(\sigma)<_X\vartheta(t)$.
  \item We have $\supp^T_X(\sigma)\lef_X\vartheta(\sigma)$.
 \end{enumerate}
 If such a function exists, then $X$ is called a Bachmann-Howard fixed point of $T$.
\end{definition}

In general Bachmann-Howard fixed points do not need to be well founded, but we are particularly interested in the case where they are:

\begin{definition}[$\prs$]\label{def:abstract-bhp}
 The abstract Bachmann-Howard principle is the assertion that every dilator has a well-founded Bachmann-Howard fixed point.
\end{definition}

Let us point out that the side condition $\supp^T_X(\sigma)\lef_X\vartheta(\tau)$ in the definition of a Bachmann-Howard collapse is crucial: It is possible that the order-type of $T_X$ is bigger than the order-type of $X$, for any well-order $X$. In this case the function $\vartheta:T_X\rightarrow X$ cannot be fully order preserving. The definition of Bachmann-Howard collapse is inspired by the construction of the Bachmann-Howard ordinal, in particular by the notation system due to Rathjen (see~\cite[Section~1]{rathjen-weiermann-kruskal}). Using a strong meta theory, it is standard to show that the abstract Bachmann-Howard principle is sound (e.g.~the first uncountable cardinal is a Bachmann-Howard fixed point of any dilator with hereditarily countable parameters, see~\cite[Section~2]{freund-equivalence}). On the other hand, the fact that a Bachmann-Howard collapse $\vartheta:T_X\rightarrow X$ is ``almost'' order preserving requires the existence of rather large ordinals, which suggests that the abstract Bachmann-Howard principle is strong. This is confirmed by the following result:

\begin{theorem}\label{thm:main-abstract}
 The following are equivalent over $\atrs$:
 \begin{enumerate}[label=(\roman*)]
  \item The principle of $\Pi^1_1$-comprehension.
  \item The statement that every set is an element of some admissible set.
  \item The abstract Bachmann-Howard principle.
 \end{enumerate}
\end{theorem}

The equivalence between (i) and (ii) is shown in \cite[Section~7]{jaeger-admissibles} (see also~\cite[Section~1.4]{freund-thesis}, where the base theory is lowered to $\atrs$). The equivalence between (ii) and (iii) is established in~\cite{freund-equivalence}, based on similar results in \cite{freund-bh-preprint,freund-thesis}. Note that Girard, in the unpublished second part of his book on proof theory~\cite[Section~11.6]{girard-book-part2}, states a related equivalence but does not give a complete proof.

The aim of the present paper is to resolve two shortcomings of Theorem~\ref{thm:main-abstract}: First we will show that the equivalence between (i) and (iii) holds over the base theory~$\rca$. For this purpose we formalize dilators in second-order arithmetic. This was already done by Girard, but we find it worthwhile to give a detailed presentation in terms of support functions (cf.~part~(ii) of Definition~\ref{def:prae-dilator}). The formalization relies on Girard's result that a prae-dilator is determined by its restriction to the category of natural numbers (up to natural isomorphism). In fact we will see that there is a single primitive recursive set function that reconstructs any prae-dilator from its set-sized restriction. This has two welcome side effects: It will enable us to express the abstract Bachmann-Howard principle by a single formula rather than a schema. And it means that the restriction to primitive recursive set functions is no loss of generality, as promised above. Once statement~(iii) of Theorem~\ref{thm:main-abstract} is expressed in second-order arithmetic it is immediate that the equivalence between (i) and (iii) holds over the second-order theory $\atr$, over which $\atrs$ is conservative (due to Simpson~\cite{simpson82,simpson09}). We then prove that the abstract Bachmann-Howard principle implies arithmetical transfinite recursion. It follows that the equivalence between (i) and (iii) holds over~$\rca$.

The second shortcoming of Theorem~\ref{thm:main-abstract} results from our ``abstract'' formulation of the Bachmann-Howard principle: We have merely asserted the existence of a Bachmann-Howard fixed point, without specifying how a concrete fixed point might be constructed. Thus it is not immediately clear whether the strength of the abstract Bachmann-Howard principle lies in the existence of a Bachmann-Howard fixed point or in the assertion that such a fixed point is well-founded. In~\cite{freund-categorical} (similarly already in~\cite[Section~2.2]{freund-thesis}) we have shown that any prae-dilator $T$ has a minimal Bachmann-Howard fixed point $\bh(T)$, which can be constructed by a primitive recursive set function: The idea is to define $\bh(T)$ as the direct limit of orders $X_0,X_1,\dots$. The construction ensures that we have almost order preserving collapsing functions $\vartheta_{X_n}:T_{X_n}\rightarrow X_{n+1}$, which glue to the desired Bachmann-Howard collapse $\vartheta:T_{\bh(T)}\rightarrow\bh(T)$. In the present paper we give a construction that can be carried out in~$\rca$: We describe a notation system $\vartheta(T)$ for $\bh(T)$, which is computable relative to a given prae-dilator $T$ (or rather, relative to the restriction of $T$ to the category of natural numbers). We can then state a computable Bachmann-Howard principle, which asserts that $\vartheta(T)$ is well-founded for any dilator $T$. Due to the minimality of $\vartheta(T)$ the computable Bachmann-Howard principle is equivalent to its abstract counterpart and thus to $\Pi^1_1$-comprehension.

Let us explain why the results of the present paper are important: A type-one well-ordering principle is a (computable) transformation $X\mapsto T_X$ of linear orders, together with the assertion that $T_X$ is well-founded for any well-order $X$. Observe that such an assertion is a $\Pi^1_2$-statement. The literature contains many equivalences between type-one well-ordering principles and natural $\Pi^1_2$-statements that are known from reverse mathematics (see~\cite{girard87,hirst94,marcone-montalban,rathjen-afshari,friedman-mw,rathjen-weiermann-atr,marcone-montalban,rathjen-atr,rathjen-model-bi,thomson-thesis,thomson-rathjen-Pi-1-1}). Rathjen~\cite{rathjen-wops-chicago,rathjen-atr} and Montalb\'an~\cite{montalban-draft,montalban-open-problems} had conjectured that $\Pi^1_1$-comprehension, which is a $\Pi^1_3$-statement, has a similar characterization by a type-two well-ordering principle. Such a principle should transform each type-one well-ordering principle into a well-order (or into another type-one well-ordering principle, but the type of the codomain can be lowered by Currying). Theorem~\ref{thm:main-abstract} makes a huge step towards that conjecture: The abstract Bachmann-Howard principle does certainly encapsulate a type-two well-ordering principle. It fails, however, to separate the well-ordering principle into a computable transformation and a statement about the preservation of well-foundedness. The computable Bachmann-Howard principle achieves this separation, so that we finally have a fully satisfactory solution of Rathjen and Montalb\'an's conjecture.

To conclude this introduction the author would like to point out that parts of the present paper are based on Sections~2.3 and~2.4 of his PhD thesis~\cite{freund-thesis}.

\section{Dilators in Primitive Recursive Set Theory and Second-Order Arithmetic}\label{sect:dilator-prs-so}

Girard~\cite{girard-pi2} has shown that dilators are determined by their restrictions to the category of natural numbers. This makes it possible to view them as set-sized objects and to represent them in second-order arithmetic. The aim of the present section is to give a more explicit presentation of the constructions that are involved (the support functions from Definition~\ref{def:prae-dilator} will turn out very useful). As a result, we will see that the abstract Bachmann-Howard principle can be expressed by a single formula (rather than a schema) in the language of second-order arithmetic (rather than set theory).

In the first half of this section we work in primitive recursive set theory ($\prs$). Our goal is to define a primitive recursive set function that reconstructs any prae-dilator from its restriction to the category of natural numbers. The objects of this category are the natural numbers, each identified with its ordered set of predecessors. The morphisms are the order embeddings
\begin{equation*}
 n=\{0,\dots,n-1\}\rightarrow\{0,\dots,m-1\}=m.
\end{equation*}
Note that the resulting category is equivalent to the category of finite linear orders. We write $\en_a:|a|\rightarrow a$ for the isomorphism between a finite linear order $a$ and its cardinality $|a|$. If $f:a\rightarrow b$ is an embedding of finite linear orders, then $|f|:|a|\rightarrow|b|$ denotes the unique order preserving function with
\begin{equation*}
\en_b\circ|f|=f\circ\en_a. 
\end{equation*}
Thus $\en_{(\cdot)}$ is a natural isomorphism between the functor $|\cdot|$ and the identity. Let us also fix the notation
\begin{equation*}
 \iota_X^Y:X\hookrightarrow Y
\end{equation*}
for the inclusion of sets $X\subseteq Y$. We will show that prae-dilators are equivalent to the following set-sized objects:

\begin{definition}[$\prs$]\label{def:set-prae-dilator}
 A set-sized prae-dilator consists of
 \begin{enumerate}[label=(\roman*)]
  \item a functor $T$ from natural numbers to linear orders and
  \item a natural transformation $\supp^T:T\Rightarrow[\cdot]^{<\omega}$ which computes supports, in the following sense: For any $n$ and $\sigma\in T_n$ we have $\sigma\in\rng(T_{\iota_\sigma\circ\en_\sigma})$, with
  \begin{alignat*}{3}
   \en_\sigma&=\en_{\supp^T_n(\sigma)}&&:{}&|\supp^T_n(\sigma)|&\rightarrow\supp^T_n(\sigma),\\
   \iota_\sigma&=\iota_{\supp^T_n(\sigma)}^n&&:{}&\supp^T_n(\sigma)&\hookrightarrow\{0,\dots,n-1\}.
  \end{alignat*}
 \end{enumerate}
\end{definition}

The definition of set-sized dilator will be given below, as it requires some preliminary work (it will not be enough to test the well-foundedness of $T_n$ for all natural numbers). Note that we could not demand $\sigma\in\rng(T_{\iota_\sigma})$ in condition (ii) above, since $\supp^T_n(\sigma)$ may not be a natural number. If $T$ is a prae-dilator in the sense of Definition~\ref{def:prae-dilator} (in the following we will speak of class-sized (prae-)dilators), then $\sigma\in\rng(T_{\iota_\sigma})$ implies $\sigma\in\rng(T_{\iota_\sigma\circ\en_\sigma})$, because $\en_\sigma$ is an isomorphism. This gives one direction of our desired equivalence:

\begin{lemma}[$\prs$]\label{lem:prae-dilator-restriction-set}
 The restriction of a class-sized prae-dilator $T$ to the category of natural numbers yields a set-sized prae-dilator~$T\!\restriction\!\mathbb N$.
\end{lemma}

For the other direction we must reconstruct $T$ from $T\!\restriction\!\mathbb N$. The idea is to use the pair $\langle a,\sigma\rangle$ with $a\in[X]^{<\omega}$ and $\sigma\in T_{|a|}$ to represent the element $T_{\iota_a^X\circ\en_a}(\sigma)\in T_X$. To get a unique representation we include the minimality condition $\supp^T_{|a|}(\sigma)=|a|$.

\begin{definition}[$\prs$]\label{def:reconstruct-dilators}
Let $(T,\supp^T)$ be a set-sized prae-dilator. For each linear order $X$ we define a set $D^T_X$ and a binary relation $<_{D^T_X}$ on $D^T_X$ by
\begin{gather*}
 D^T_X:=\{\langle a,\sigma\rangle\,|\,a\in[X]^{<\omega}\text{ and }\sigma\in T_{|a|}\text{ and }\supp^T_{|a|}(\sigma)=|a|\},\\
 \langle a,\sigma\rangle<_{D^T_X}\langle b,\tau\rangle:\Leftrightarrow T_{|\iota_a^{a\cup b}|}(\sigma)<_{T_{|a\cup b|}}T_{|\iota_b^{a\cup b}|}(\tau).
\end{gather*}
If $f:X\rightarrow Y$ is an order embedding, then we put
\begin{equation*}
 D^T_f(\langle a,\sigma\rangle):=\langle [f]^{<\omega}(a),\sigma\rangle
\end{equation*}
to define a function $D^T_f:D^T_X\rightarrow D^T_Y$ (note $|[f]^{<\omega}(a)|=|a|$ to see $D^T_f(\langle a,\sigma\rangle)\in D^T_Y$). Finally, we define a family of functions $\supp^{D^T}_X:D^T_X\rightarrow[X]^{<\omega}$ by setting
\begin{equation*}
 \supp^{D^T}_X(\langle a,\sigma\rangle):=a
\end{equation*}
for each linear order $X$.
\end{definition}

It is straightforward to check that the maps $(T,X)\mapsto(D^T_X,<_{D^T_X})$, $(T,f)\mapsto D^T_f$ and $(T,X)\mapsto\supp^{D^T}_X$ are primitive recursive set functions (see~\cite{freund-thesis} for details). For the other direction of our equivalence we show the following:

\begin{lemma}[$\prs$]\label{lem:set-prae-dilator-to-class}
 If $(T,\supp^T)$ is a set-sized prae-dilator, then $(D^T,\supp^{D^T})$ is a class-sized prae-dilator.
\end{lemma}
\begin{proof}
 In order to verify that $(D^T_X,<_{D^T_X})$ is a linear order one needs the implication
 \begin{equation*}
  T_{|\iota_a^{a\cup b}|}(\sigma)=T_{|\iota_b^{a\cup b}|}(\tau)\quad\Rightarrow\quad\langle a,\sigma\rangle=\langle b,\tau\rangle.
 \end{equation*}
 The naturality of $\supp^T$ and the condition $\supp^T_{|a|}(\sigma)=|a|$ imply
 \begin{multline*}
[\en_{a\cup b}]^{<\omega}\circ\supp^T_{|a\cup b|}\circ T_{|\iota_a^{a\cup b}|}(\sigma)=\\
=[\en_{a\cup b}]^{<\omega}\circ[|\iota_a^{a\cup b}|]^{<\omega}\circ\supp^T_{|a|}(\sigma)=[\iota_a^{a\cup b}\circ\en_a]^{<\omega}(|a|)=a.
\end{multline*}
As $a$ is determined by $T_{|\iota_a^{a\cup b}|}(\sigma)$, the assumption $T_{|\iota_a^{a\cup b}|}(\sigma)=T_{|\iota_b^{a\cup b}|}(\tau)$ yields $a=b$. Then $|\iota_a^{a\cup b}|=|\iota_b^{a\cup b}|$ is the identity on $|a|=|a\cup b|=|b|$, and we also get $\sigma=\tau$. Based on this fact it is straightforward to verify that $D^T$ is an endofunctor of linear orders and that $\supp^{D^T}:D^T\Rightarrow[\cdot]^{<\omega}$ is a natural transformation. It remains to show that $\supp^{D^T}$ computes supports: Observe that $\langle a,\sigma\rangle\in D^T_X$ implies $\langle a,\sigma\rangle\in D^T_a$. Writing $\iota_{\langle a,\sigma\rangle}:\supp^{D^T}_X(\langle a,\sigma\rangle)=a\hookrightarrow X$ for the inclusion we have
\begin{equation*}
 D^T_{\iota_{\langle a,\sigma\rangle}}(\langle a,\sigma\rangle)=\langle [\iota_{\langle a,\sigma\rangle}]^{<\omega}(a),\sigma\rangle=\langle a,\sigma\rangle,
\end{equation*}
which confirms that $\langle a,\sigma\rangle$ lies in the range of $D^T_{\iota_{\langle a,\sigma\rangle}}$.
\end{proof}

Let us show that we have reconstructed the original dilator:

\begin{proposition}[$\prs$]\label{prop:dilator-reconstruct}
 For any class-sized prae-dilator $(T,\supp^T)$ we can construct a natural equivalence $\eta^T:D^{T\restriction\mathbb N}\Rightarrow T$ with $\supp^T_X\circ\eta^T_X=\supp^{D^{T\restriction\mathbb N}}_X$.
\end{proposition}
\begin{proof}
 We make the above intuition official and set
 \begin{equation*}
 \eta^T_X(\langle a,\sigma\rangle):=T_{\iota_a^X\circ\en_a}(\sigma).                                            
 \end{equation*}
 Concerning the formalization in $\prs$, note that a primitive recursive definition of~$T$ is readily transformed into a primitive recursive definition of $\eta^T$. We verify that $\eta^T_X:D^{T\restriction\mathbb N}_X\rightarrow T_X$ is order preserving: Assume that we have $\langle a,\sigma\rangle<_{D^{T\restriction\mathbb N}_X}\langle b,\tau\rangle$ and thus $T_{|\iota_a^{a\cup b}|}(\sigma)<_{T_{|a\cup b|}}T_{|\iota_b^{a\cup b}|}(\tau)$. In view of
 \begin{equation*}
  \iota_a^X\circ\en_a=\iota_{a\cup b}^X\circ\iota_a^{a\cup b}\circ\en_a=\iota_{a\cup b}^X\circ\en_{a\cup b}\circ|\iota_a^{a\cup b}|
 \end{equation*}
we obtain the desired inequality
\begin{multline*}
 \eta^{T}_X(\langle a,\sigma\rangle)=T_{\iota_a^X\circ\en_a}(\sigma)=T_{\iota_{a\cup b}^X\circ\en_{a\cup b}}\circ T_{|\iota_a^{a\cup b}|}(\sigma)<_{T_X}{}\\
{}<_{T_X} T_{\iota_{a\cup b}^X\circ\en_{a\cup b}}\circ T_{|\iota_b^{a\cup b}|}(\tau)=T_{\iota_b^X\circ\en_b}(\tau)=\eta^T_X(\langle b,\tau\rangle).
\end{multline*}
To establish naturality we consider $f:X\rightarrow Y$ and observe
\begin{equation*}
 f\circ\iota_a^X\circ\en_a=\iota_{[f]^{<\omega}(a)}^Y\circ(f\!\restriction\!a)\circ\en_a=\iota_{[f]^{<\omega}(a)}^Y\circ\en_{[f]^{<\omega}(a)}.
\end{equation*}
This does indeed yield
\begin{multline*}
 T_f\circ\eta^T_X(\langle a,\sigma\rangle)=T_{f\circ\iota_a^X\circ\en_a}(\sigma)=T_{\iota_{[f]^{<\omega}(a)}^Y\circ\en_{[f]^{<\omega}(a)}}(\sigma)=\\
=\eta^T_Y(\langle[f]^{<\omega}(a),\sigma\rangle)=\eta^T_Y\circ D^{T\restriction\mathbb N}_f(\langle a,\sigma\rangle).
\end{multline*}
Next, we show that the functions $\eta^T_X:D^{T\restriction\mathbb N}_X\rightarrow T_X$ are surjective: By the definition of prae-dilator any $\sigma\in T_X$ lies in the range of $T_{\iota_a^X}$, for $a:=\supp^T_X(\sigma)$. Since the function $\en_a:|a|\rightarrow a$ is an isomorphism we obtain a $\sigma_0\in T_{|a|}$ with $\sigma=T_{\iota_a^X\circ\en_a}(\sigma_0)$. To conclude $\sigma\in\rng(\eta^T_X)$ it remains to verify $\langle a,\sigma_0\rangle\in D^{T\restriction\mathbb N}_X$. The crucial condition $\supp^{T\restriction\mathbb N}_{|a|}(\sigma_0)=|a|$ holds in view of
\begin{equation*}
 [\iota_a^X\circ\en_a]^{<\omega}\circ\supp^T_{|a|}(\sigma_0)=\supp^T_X\circ T_{\iota_a^X\circ\en_a}(\sigma_0)=\supp^T_X(\sigma)=a.
\end{equation*}
So far we have shown that $\eta^T$ is a natural isomorphism. Finally, for $\langle a,\sigma\rangle\in D^{T\restriction\mathbb N}_X$ the condition $\supp^{T\restriction\mathbb N}_{|a|}(\sigma)=|a|$ implies
\begin{multline*}
 \supp^T_X\circ\eta^T_X(\langle a,\sigma\rangle)=\supp^T_X\circ T_{\iota_a^X\circ\en_a}(\sigma)=[\iota_a^X\circ\en_a]^{<\omega}\circ\supp^T_{|a|}(\sigma)=\\
=[\iota_a^X\circ\en_a]^{<\omega}(|a|)=a=\supp^{D^{T\restriction\mathbb N}}_X(\langle a,\sigma\rangle),
\end{multline*}
as was promised in the proposition.
\end{proof}

Later we will also need the following:

\begin{lemma}[$\prs$]\label{lem:set-sized-dils-isomorphic}
 Consider set-sized prae-dilators $S$ and $T$. Given a natural equivalence $\eta^0:S\Rightarrow T$ with $\supp^T_n\circ\eta^0_n=\supp^S_n$, we can construct a natural equivalence $\eta:D^S\Rightarrow D^T$ with $\supp^{D^T}_X\circ\eta_X=\supp^{D^S}_X$.
\end{lemma}
\begin{proof}
 It is straightforward to verify that
 \begin{equation*}
  \eta_X(\langle a,\sigma\rangle):=\langle a,\eta^0_{|a|}(\sigma)\rangle
 \end{equation*}
 defines the desired family of functions.
\end{proof}

Whether a given set is a set-sized prae-dilator can be decided by a primitive recursive set function. In contrast, the notion of dilator retains its logical complexity:

\begin{definition}[$\prs$]
 A set-sized prae-dilator $(T,\supp^T)$ is called a set-sized dilator if the order $(D^T_X,<_{D^T_X})$ is well-founded for any well-order $X$.
\end{definition}

Proposition~\ref{prop:dilator-reconstruct} does, in particular, tell us that $D^{T\restriction\mathbb N}_X$ is well-founded if $T_X$ is. Together with Lemma~\ref{lem:prae-dilator-restriction-set} we get the following:

\begin{corollary}[$\prs$]
 If $T$ is a class-sized dilator, then its restriction $T\!\restriction\!\mathbb N$ is a set-sized dilator.
\end{corollary}

The converse is trivial, based on the corresponding result for prae-dilators:

\begin{corollary}[$\prs$]
 If $(T,\supp^T)$ is a set-sized dilator, then $(D^T,\supp^{D^T})$ is a class-sized dilator.
\end{corollary}

The previous results show that (prae-)dilators are essentially set-sized objects. As promised, this allows us to express the abstract Bachmann-Howard principle by a single formula (recall that Definition~\ref{def:abstract-bhp} is a schema, because we cannot quantify over all class-sized dilators).

\begin{proposition}[$\prs$]\label{prop:bh-single-instance}
 The following consequence of the abstract Bachmann-Howard principle implies each of its instances: For every set-sized dilator $T$ there is a well-order~$X$ with a Bachmann-Howard collapse $\vartheta:D^T_X\rightarrow X$.
\end{proposition}
\begin{proof}
 Let us first establish that the given statement follows from the Bachmann-Howard principle: If $T$ is a set-sized dilator, then $D^T$ is a class-sized dilator. Using the abstract Bachmann-Howard principle (for the primitive recursive family of functions $X\mapsto D^T_X$ with parameter $T$) we obtain the desired collapse $\vartheta:D^T_X\rightarrow X$ for a well-order $X$. Conversely, we deduce an arbitrary instance of the abstract Bachmann-Howard principle: If $T$ is a class-sized dilator, then $T\!\restriction\!\mathbb N$ is a set-sized dilator. Thus the statement from the proposition yields a well-order $X$ with a Bachmann-Howard collapse $\vartheta:D^{T\restriction\mathbb N}_X\rightarrow X$. Proposition~\ref{prop:dilator-reconstruct} provides an isomorphism $\eta^T_X:D^{T\restriction\mathbb N}_X\rightarrow T_X$ with $\supp^T_X\circ\eta^T_X=\supp^{D^{T\restriction\mathbb N}}_X$. One can check that
 \begin{equation*}
  \vartheta\circ(\eta^T_X)^{-1}:T_X\rightarrow X
 \end{equation*}
 is a Bachmann-Howard collapse as well.
\end{proof}

As promised, we can also deduce that the restriction to primitive recursive set functions does not mean a loss of generality:

\begin{remark}
 If one works in a stronger theory, then one may want to consider a dilator $T$ that is not given by a primitive recursive set function. The corresponding dilator $D^{T\restriction\mathbb N}$ will still be primitive recursive (even though stronger separation axioms may be needed to show that the parameter $T\!\restriction\!\mathbb N$ exists as a set). Furthermore, the dilators $T$ and $D^{T\restriction\mathbb N}$ will still be equivalent (even though the equivalence $\eta^T:D^{T\restriction\mathbb N}\Rightarrow T$ may no longer be primitive recursive). As in the previous proposition, a Bachmann-Howard collapse for $D^{T\restriction\mathbb N}$ can be transformed into a Bachmann-Howard collapse for $T$. This shows that the abstract Bachmann-Howard principle does not become stronger if we admit dilators which are not primitive recursive.
\end{remark}

While any dilator is equivalent to a primitive recursive one, the statement that ``all dilators are primitive recursive set functions'' can also be very misleading:

\begin{remark}\label{rmk:dilators-order-types}
 For each dilator $T$ we can consider the corresponding function
 \begin{equation*}
  \alpha\mapsto\otyp(T_\alpha)
 \end{equation*}
of  order types. It is important to realize that this function is not primitive recursive in general. By induction on the primitive recursive definition of $F$ we find a number $n$ such that $\vec x\in\mathbb V_\beta$ implies $F(\vec x)\in\mathbb V_{\varphi_n(\beta)}$. Thus $\alpha\mapsto\varphi_\alpha(0)$ cannot be a primitive recursive set function. On the other hand one can construct primitive recursive notation systems~$T_X$ such that $\otyp(X)=\alpha$ implies $\otyp(T_X)=\varphi_\alpha0$ (see~\mbox{\cite{marcone-montalban,rathjen-weiermann-atr}}). To explain this phenomenon we recall that the order type of a given well-order cannot be computed by a primitive recursive set function (axiom beta is not provable in~$\prs$ and not even in Kripke-Platek set theory). For this reason it is important to admit arbitrary well-orders as values of dilators. If one only allowed ordinals as values (as in the formulation of Girard~\cite[Definition~2.3.1]{girard-pi2}), then it would not be true that any dilator is equivalent to a primitive recursive one.
\end{remark}

In the second half of this section we show that (prae-)dilators can be formalized in the subsystem $\rca$ of second-order arithmetic. This is due to Girard~\cite{girard-pi2}, but we know of no explicit presentation. It is well-known that finite sets of natural numbers and functions between such sets can be coded by natural numbers. Basic relations and operations on the codes are primitive recursive (in the usual number-theoretic sense). This allows us to express the following by an arithmetical formula (with parameters $T^0,T^1,\supp^T\subseteq\mathbb N$):

\begin{definition}[$\rca$]\label{def:coded-prae-dilator}
 A coded prae-dilator consists of
 \begin{enumerate}[label=(\roman*)]
  \item a functor $T$ from natural numbers to linear orders with fields $T_n\subseteq\mathbb N$, represented by the sets
  \begin{align*}
  T^0&=\{\langle 0,n,\sigma\rangle\,|\,\sigma\in T_n\}\cup\{\langle 1,n,\sigma,\tau\rangle\,|\,\sigma<_{T_n}\tau\},\\
  T^1&=\{\langle f,\sigma,\tau\rangle\,|\,T_f(\sigma)=\tau\},
  \end{align*}
  \item a natural transformation $\supp^T:T\Rightarrow[\cdot]^{<\omega}$ that computes supports (in the sense of Definition~\ref{def:set-prae-dilator}), represented by the set
  \begin{equation*}
  \supp^T=\{\langle n,\sigma,a\rangle\,|\,\supp^T_n(\sigma)=a\}.
  \end{equation*}
 \end{enumerate}
\end{definition}

When we work in $\rca$ we can only refer to the sets $T^0,T^1,\supp^T$. In this context we use $\sigma\in T_n$ as an abbreviation for the $\Delta^0_1$-formula~$\langle 0,n,\sigma\rangle\in T^0$. The same applies to the expressions $\sigma<_{T_n}\tau$, $T_f(\sigma)=\tau$ and $\supp^T_n(\sigma)=a$. Formulas of second-order arithmetic have a natural translation into the language of set theory (cf.~\cite[Theorem~VII.3.9]{simpson09}). If $T^0$ and $n\mapsto T_n$ are related as in the definition, then the set-theoretic translation of the second-order formula $\sigma\in T_n\equiv\langle 0,n,\sigma\rangle\in T^0$ is equivalent to the set-theoretic formula $\sigma\in T_n$ (even though the two formulas are not literally equal). When we work in $\prs$ we may thus identify coded prae-dilators and those set-sized prae-dilators $(T,\supp^T)$ with the property that the field of each linear order~$T_n$ is a subset of the natural numbers.

Our next goal is to single out the coded dilators in $\rca$. For this purpose we reconstruct the orders $(D^T_X,<_{D^T_X})$ from Definition~\ref{def:reconstruct-dilators}, for each order $(X,<_X)$ with field~$X\subseteq\mathbb N$: The set
\begin{equation*}
 D^T_X=\{\langle a,\sigma\rangle\,|\,\text{``$a$ codes a finite subset of $X$''}\land\sigma\in T_{|a|}\land\supp^T_{|a|}(\sigma)=|a|\}\subseteq\mathbb N
\end{equation*}
exists by $\Delta^0_1$-comprehension. Recall the functions $|\iota_c^d|:|c|\rightarrow|d|$ for finite subsets (in fact suborders) $c,d$ of $X$. It is straightforward to see that the operation $(c,d)\mapsto|\iota_c^d|$ on the codes is primitive recursive relative to $(X,<_X)$. In view of
\begin{align*}
 \langle a,\sigma\rangle<_{D^T_X}\langle b,\tau\rangle&\equiv\exists_{\sigma',\tau'}(T_{|\iota_a^{a\cup b}|}(\sigma)=\sigma'\land T_{|\iota_b^{a\cup b}|}(\tau)=\tau'\land\sigma'<_{T_{|a\cup b|}}\tau')\\
 &\equiv\forall_{\sigma',\tau'}(T_{|\iota_a^{a\cup b}|}(\sigma)=\sigma'\land T_{|\iota_b^{a\cup b}|}(\tau)=\tau'\rightarrow\sigma'<_{T_{|a\cup b|}}\tau')
\end{align*}
the relation $<_{D^T_X}$ can be defined by $\Delta^0_1$-comprehension as well. As in the proof of Lemma~\ref{lem:set-prae-dilator-to-class} one can show the following:

\begin{lemma}[$\rca$]
 If $T$ is a coded prae-dilator and $(X,<_X)$ is a linear order, then $(D^T_X,<_{D^T_X})$ is a linear order as well.
\end{lemma}

Coded dilators can now be defined by a $\Pi^1_2$-formula of second-order arithmetic:

\begin{definition}[$\rca$]\label{def:coded-dilator}
 A coded prae-dilator $T$ is a coded dilator if $(D^T_X,<_{D^T_X})$ is well-founded for every well-order $(X,<_X)$ with field $X\subseteq\mathbb N$.
\end{definition}

There is one subtlety when we translate back into set theory: The previous definition does only probe well-orders $(X,<_X)$ with~$X\subseteq\mathbb N$. In a set theory with choice this does not make a difference: Using the definition of well-foundedness in terms of descending sequences, one can show that $D^T_Y$ is well-founded if $D^T_{Y_0}$ is well-founded for all countable suborders $Y_0\subseteq Y$ (due to Girard~\cite[Theorem~2.1.15]{girard-pi2}). All countable orders are covered by Definition~\ref{def:coded-dilator}, since Lemma~\ref{lem:set-prae-dilator-to-class} tells us that $Y_0\cong X\subseteq\mathbb N$ implies $D^T_{Y_0}\cong D^T_X$. In the base theory $\atrs$ of Theorem~\ref{thm:main-abstract} the issue disappears for a rather different reason: This theory includes the axiom of countability, which implies that any set is in bijection with a subset of the natural numbers (axiom beta, which is also included in $\atrs$, is not needed here).

Back in $\rca$, a Bachmann-Howard collapse $\vartheta:D^T_X\rightarrow X$ for an order $X\subseteq\mathbb N$ can be represented by a set $\vartheta\subseteq\mathbb N$. The conditions from Definition~\ref{def:bachmann-howard-collapse} are readily expressed by an arithmetical formula (with the appropriate set parameters). Thus the following amounts to a $\Pi^1_3$-statement of second-order arithmetic:

\begin{definition}[$\rca$]\label{def:so-bh}
The second-order version of the abstract Bachmann-Howard principle is the following statement: For every coded dilator $T$ there is a well-founded Bachmann-Howard fixed point $X$, i.e.~a well-order $X\subseteq\mathbb N$ with a Bachmann-Howard collapse $\vartheta:D^T_X\rightarrow X$.
\end{definition}

 We will speak of the set-theoretic version of the abstract Bachmann-Howard principle in order to refer to Definition~\ref{def:abstract-bhp}. In an appropriate meta theory one can construct prae-dilators with uncountable parameters that only have uncountable Bachmann-Howard fixed points: It is straightforward to show that any Bachmann-Howard collapse must be injective (cf.~\cite[Lemma~2.1.7]{freund-thesis}). Now consider the constant prae-dilator $T_X=Y$ for an uncountable order $Y$. On the other hand, the axiom of countability ensures the following:

\begin{lemma}[$\atrs$]\label{lem:abstract-bhp-set-so}
The second-order version of the abstract Bachmann-Howard principle is equivalent to the set-theoretic version.
\end{lemma}
\begin{proof}
 To deduce the second-order version from the set-theoretic version we consider a coded dilator $T$. In the presence of countability we may view $T$ as a set-sized dilator, as discussed above. Then the set-theoretic version of the Bachmann-Howard principle yields a Bachmann-Howard collapse $\vartheta:D^T_X\rightarrow X$ for some well-order $X$. The axiom of countability yields a well-order $Y\cong X$ with field $Y\subseteq\mathbb N$. It is straightforward to transform $\vartheta$ into a Bachmann-Howard collapse of $D^T_Y$ into $Y$, as demanded by the second-order version of the abstract Bachmann-Howard principle. To deduce the set-theoretic version from the second-order version we invoke Proposition~\ref{prop:bh-single-instance} and consider a set-sized dilator~$T$. The axiom of countability yields a bijection between $\bigcup\{T_n\,|\,n\in\mathbb N\}$ and a subset of the natural numbers. This allows us to construct a coded prae-dilator $S$ and a natural equivalence $\eta^0:S\Rightarrow T$ with $\supp^T_n\circ\eta^0_n=\supp^S_n$ (cf.~the proof of \cite[Proposition~2.3.21]{freund-thesis}). By Lemma~\ref{lem:set-sized-dils-isomorphic} we get a natural equivalence $\eta:D^S\Rightarrow D^T$ with $\supp^{D^T}_X\circ\eta_X=\supp^{D^S}_X$. In particular $D^S_X\cong D^T_X$ ensures that $S$ is a coded dilator. Now the second-order version of the abstract Bachmann-Howard principle yields a well-order $X$ with a Bachmann-Howard collapse $\vartheta:D^S_X\rightarrow X$. We can use $\eta$ to tranform $\vartheta$ into the required Bachmann-Howard collapse of $D^T_X$ into $X$, as in the proof of Proposition~\ref{prop:bh-single-instance}.
\end{proof}

In view of the previous result, Theorem~\ref{thm:main-abstract} implies that $\Pi^1_1$-comprehension is equivalent to the second-order version of the abstract Bachmann-Howard principle, over the base theory $\atrs$. Now that we have an equivalence between second-order statements we can immediately conclude that it holds over the second-order theory $\atr$, due to the conservativity result of Simpson~\cite{simpson82,simpson09}. In the next section we will show that the base theory can be lowered to $\rca$.

\section{Bootstrapping the Bachmann-Howard Principle}\label{sect:bootstrapping}

Via a series of intermediate steps we show that (the second-order version of) the abstract Bachmann-Howard principle implies arithmetical transfinite recursion. This will immediately allow us to lower the base theory in Theorem~\ref{thm:main-abstract}.

To initiate our bootstrapping process we prove that the abstract Bachmann-Howard principle makes ordinal exponentiation available (on the level of notation systems). By a result of Girard~\cite[Section~II.5]{girard87} (see also the proof by Hirst~\cite{hirst94}) this brings us up to $\aca$. Writing $X^{<\omega}$ for the set of finite sequences with entries in a set $X\subseteq\mathbb N$ (coded by natural numbers), we consider the following structure:

\begin{definition}[$\rca$]
 For each linear order $(X,<_X)$ the set $\omega^X\subseteq X^{<\omega}$ and the relation ${<_{\omega^X}}\subseteq\omega^X\times\omega^X$ are given by
 \begin{align*}
  \langle x_0,\dots,x_{n-1}\rangle\in\omega^X&\Leftrightarrow x_{n-1}\leq_X\dots\leq_X x_0,\\
  \langle x_0,\dots,x_{n-1}\rangle<_{\omega^X}\langle x'_0,\dots,x'_{m-1}\rangle&\Leftrightarrow\begin{cases}
                                                                                                 \text{either $n<m$ and $x_i=x'_i$ for $i<n$,}\\[1ex]
                                                                                                 \parbox{.4\textwidth}{or there is $j<\min\{n,m\}$ with\newline $x_j<_Xx'_j$ and $x_i=x_i'$ for $i<j$.}
                                                                                               \end{cases}
 \end{align*}
\end{definition}

On an informal level, assume that $X$ is isomorphic to an ordinal $\alpha$. If the elements $x_{n-1}\leq_X\dots\leq_X x_0$ correspond to $\alpha_{n-1}\leq\dots\leq\alpha_0<\alpha$, then $\langle x_0,\dots,x_{n-1}\rangle\in\omega^X$ represents the ordinal $\omega^{\alpha_0}+\dots+\omega^{\alpha_{n-1}}<\omega^\alpha$ in Cantor normal form. Thus $\omega^X$ is a notation system for the ordinal $\omega^\alpha$. The following is standard:

\begin{lemma}[$\rca$]
 If $(X,<_X)$ is a linear order, then so is $(\omega^X,<_{\omega^X})$.
\end{lemma}

The aforementioned result of Girard and Hirst implies that $\rca$ cannot show the well-foundedness of $\omega^X$. In the following we refer to the second-order version of the abstract Bachmann-Howard principle, as formulated in Definition~\ref{def:so-bh}.

\begin{proposition}[$\rca$]\label{prop:exp-well-founded}
 The abstract Bachmann-Howard principle implies that $\omega^X$ is well-founded for any well-order $X$.
\end{proposition}
\begin{proof}
 For a fixed well-order $X$ we consider the dilator $(T,\supp^T)$ given by
 \begin{align*}
  T_Y&=X\times(\{\bot\}\cup Y),\\[1ex]
  (x,y)<_{T_Y}(x',y')&\Leftrightarrow\begin{cases}
                                    \text{either $x<_Xx'$,}\\[1ex]
                                    \text{or $x=x'$ and $y<_Yy'$ (with $\bot<_Yy'$ for any $y'\in Y$),}
                                   \end{cases}\\[1ex]
  T_f(x,y)&=\begin{cases}
              (x,f(y)) & \text{if $y\in Y$ (where $f:Y\rightarrow Y'$ is an embedding),}\\[1ex]
              (x,\bot) & \text{if $y=\bot$,}
            \end{cases}\\[1ex]
  \supp^T_Y(x,y)&=\begin{cases}
                  \{y\} & \text{if $y\in Y$,}\\[1ex]
                  \emptyset & \text{if $y=\bot$.}
                \end{cases}
 \end{align*}
 Officially we must work with the representations $T^0,T^1$ and $\supp^T$ from Definition~\ref{def:coded-prae-dilator}: It is straightforward to see that they exist as sets (by $\Delta^0_1$-comprehension) and that they represent a coded prae-dilator. Instead of $T_Y$ we must consider
 \begin{equation*}
  D^T_Y=\{\left\langle\emptyset,(x,\bot)\right\rangle\,|\,x\in X\}\cup\{\left\langle\{y\},(x,0)\right\rangle\,|\,x\in X,y\in Y\}.
 \end{equation*}
Note that $(x,\bot)\in T_{|\emptyset|}$ satisfies $\supp^T_{|\emptyset|}(x,\bot)=\emptyset=|\emptyset|$ while $(x,0)\in T_{|\{y\}|}$ satisfies $\supp^T_{|\{y\}|}(x,0)=\{0\}=1=|\{y\}|$, as demanded by the definition of $D^T_Y$ in the previous section. Since the interpretation
\begin{equation*}
 D^T_Y\ni\left\langle\emptyset,(x,\bot)\right\rangle\mapsto(x,\bot)\in T_Y,\qquad
 D^T_Y\ni\left\langle\{y\},(x,0)\right\rangle\mapsto(x,y)\in T_Y
\end{equation*}
is an isomorphism (with respect to the order $<_{D^T_Y}$ from the previous section) we may work with~$T_Y$ rather than $D^T_Y$ after all. To see that $T$ is a coded dilator we must show that $D^T_Y\cong T_Y$ is well-founded for any well-order $Y$. Aiming at a contradiction, assume that $(x_n,y_n)_{n\in\mathbb N}$ is a strictly decreasing sequence in $T_Y$ (the two obvious definitions of well-foundedness are equivalent over $\rca$, see e.g.~\cite[Lemma~2.3.12]{freund-thesis}). If the sequence $(x_n)_{n\in\mathbb N}$ does not become constant, then we can transform it into a strictly decreasing sequence in $X$. If we have $x_n=x_N$ for all $n\geq N$, then $(y_n)_{n\geq N}$ is a strictly increasing sequence in $\{\bot\}\cup Y$. Both possibilities contradict the assumption that $X$ and $Y$ are well-founded. Since $T$ is a coded dilator the abstract Bachmann-Howard principle yields a well-order $Y$ with a Bachmann-Howard collapse
\begin{equation*}
 \vartheta:T_Y\cong D^T_Y\rightarrow Y.
\end{equation*}
To deduce that $\omega^X$ is well-founded we show that the function $f:\omega^X\rightarrow\{\bot\}\cup Y$ with recursive clauses
\begin{align*}
 f(\langle\rangle)&:=\bot,\\
 f(\langle x_0,\dots,x_n\rangle)&:=\vartheta(x_0,f(\langle x_1,\dots,x_n\rangle))
\end{align*}
is order-preserving (it is worth observing that $\vartheta$ cannot be fully order-preserving: if $X$ and $Y$ have order-type $\alpha>1$ resp.~$\beta$, then $T_Y$ has order-type $(1+\beta)\cdot\alpha>\beta$). So assume that we have
\begin{equation*}
 \langle x_0,\dots,x_{n-1}\rangle<_{\omega^X}\langle x'_0,\dots,x'_{m-1}\rangle.
\end{equation*}
To see that we can cancel equal entries at the beginning of the sequences it suffices to observe that $y<_Yy'$ (possibly with $y=\bot$) implies $\vartheta(x,y)<_Y\vartheta(x,y')$: By the definition of Bachmann-Howard collapse we have $\{y'\}=\supp^T_Y(x,y')\lef_Y\vartheta(x,y')$, in other words $y'<_Y\vartheta(x,y')$. This implies $\supp^T_Y(x,y)\lef_Y\vartheta(x,y')$. Together with $(x,y)<_{T_Y}(x,y')$ we get $\vartheta(x,y)<_Y\vartheta(x,y')$, again by the definition of Bachmann-Howard collapse. It remains to consider the cases $0=n<m$ and $x_0<_X x'_0$. In the first case we observe
\begin{equation*}
 f(\langle x_0,\dots,x_{n-1}\rangle)=\bot<_Y f(\langle x'_0,\dots,x'_{m-1}\rangle)\in Y.
\end{equation*}
In case $x_0<_Xx'_0$ we set $y:=f(\langle x'_1,\dots,x'_{m-1}\rangle)$ and prove
\begin{equation*}
 f(\langle x_{n-i},\dots,x_{n-1}\rangle)<_Y\vartheta(x'_0,y)=f(\langle x'_0,\dots,x'_{m-1}\rangle)
\end{equation*}
by induction on $i\leq n$: For $i=0$ we have $f(\langle x_{n-i},\dots,x_{n-1}\rangle)=\bot$ and the claim follows as before. In the step the induction hypothesis provides
\begin{equation*}
 \supp^T_Y(x_{n-(i+1)},f(\langle x_{n-i},\dots,x_{n-1}\rangle))\subseteq\{f(\langle x_{n-i},\dots,x_{n-1}\rangle)\}\lef_Y\vartheta(x'_0,y).
\end{equation*}
Invoking the definition of $\omega^X$ we have $x_{n-(i+1)}\leq_X x_0<_Xx'_0$ and thus
\begin{equation*}
 (x_{n-(i+1)},f(\langle x_{n-i},\dots,x_{n-1}\rangle))<_{T_Y}(x'_0,y).
\end{equation*}
By the definition of Bachmann-Howard collapse this implies
\begin{equation*}
 f(\langle x_{n-(i+1)},\dots,x_{n-1}\rangle)=\vartheta(x_{n-(i+1)},f(\langle x_{n-i},\dots,x_{n-1}\rangle))<_Y\vartheta(x'_0,y),
\end{equation*}
which completes the induction step.
\end{proof}

Recall that $\varepsilon_\alpha$ denotes the $\alpha$-th ordinal with $\omega^\gamma=\gamma$. By another application of the abstract Bachmann-Howard principle we want to establish the well-foundedness of these \mbox{$\varepsilon$-numbers}. Using a result of Marcone and Montalb\'an~\cite{marcone-montalban} (see also the proof by Afshari and Rathjen~\cite{rathjen-afshari}) this will secure arithmetical recursion along the natural numbers, the defining principle of the theory $\mathbf{ACA_0^+}$. First we define a notation system $\varepsilon_X$ for $\varepsilon_\alpha$, relative to a notation system $X$ for $\alpha$. The $\varepsilon$-numbers below $\varepsilon_\alpha$ are represented by terms $\varepsilon_x$ with $x\in X$. The gaps are filled with terms of the form $\omega^{t_0}+\dots+\omega^{t_n}$, which correspond to ordinals in Cantor normal form.

\begin{definition}[$\rca$]\label{def:vareps-X}
 For each linear order $(X,<_X)$ the set $\varepsilon_X$ and the relation ${<_{\varepsilon_X}}\subseteq\varepsilon_X\times\varepsilon_X$ are defined by the following simultaneous recursion:
 \begin{enumerate}[label=(\roman*)]
  \item The term $0$ is an element of $\varepsilon_X$.
  \item For each $x\in X$ the term $\varepsilon_x$ is an element of $\varepsilon_X$.
  \item If $t_0,\dots,t_n$ are terms in $\varepsilon_X$, then so is $\omega^{t_0}+\dots+\omega^{t_n}$, provided that
  \begin{itemize}
   \item either $n=0$ and $t_0$ is not of the form $\varepsilon_x$,
   \item or $n>0$ and $t_n\leq_{\varepsilon_X}\dots\leq_{\varepsilon_X}t_0$ (note that $s\leq_{\varepsilon_X}t$ abbreviates $s<_{\varepsilon_X}t\lor s=t$, where the second disjunct refers to equality as terms).
  \end{itemize}
 \end{enumerate}
 For $t,t'\in\varepsilon_X$ we have $t<_{\varepsilon_X}t'$ precisely if one of the following holds:
 \begin{enumerate}[label=(\roman*')]
  \item We have $t=0$ and $t'\neq 0$.
  \item We have $t=\varepsilon_x$ and
  \begin{itemize}
   \item either $t'=\varepsilon_{x'}$ with $x<_X x'$,
   \item or $t'=\omega^{t'_0}+\dots+\omega^{t'_m}$ with $t\leq_{\varepsilon_X}t'_0$.
  \end{itemize}
  \item We have $t=\omega^{t_0}+\dots+\omega^{t_n}$ and
  \begin{itemize}
   \item either $t'=\varepsilon_x$ with $t_0<_{\varepsilon_X}t'$,
   \item or $t'=\omega^{t'_0}+\dots+\omega^{t'_m}$ such that one of the following holds:
   \begin{itemize}
    \item Either we have $n<m$ and $t_i=t'_i$ for all $i\leq n$,
    \item or there is $j\leq\min\{n,m\}$ with $t_j<_{\varepsilon_X}t'_j$ and $t_i=t'_i$ for $i<j$.
   \end{itemize}
  \end{itemize}
 \end{enumerate}
\end{definition}

To formalize the definition in $\rca$ one first defines a term system $\varepsilon_X^0\supseteq\varepsilon_X$ by ignoring the condition $t_n\leq_{\varepsilon_X}\dots\leq_{\varepsilon_X}t_0$ in clause~(iii). Writing $\ulcorner t\urcorner$ for the G\"odel number of the term $t\in\varepsilon_X^0$, the length function $L_{\varepsilon_X}:\varepsilon_X^0\rightarrow\mathbb N$ is given as
\begin{equation*}
 L_{\varepsilon_X}(t):=\begin{cases}
                        \ulcorner t\urcorner & \text{if $t=0$ or $t=\varepsilon_x$},\\
                        \max\{\ulcorner t\urcorner,L_{\varepsilon_X}(t_0)+\dots+L_{\varepsilon_X}(t_n)+1\} & \text{if $t=\omega^{t_0}+\dots+\omega^{t_n}$}.
                       \end{cases}
\end{equation*}
The G\"odel numbers are included to ensure that $\forall_t(L_{\varepsilon_X}(t)\leq n\rightarrow\dots)$ is a bounded quantifier. For $s,t,t'\in\varepsilon_X^0$ one can now decide $s\in\varepsilon_X$ and $t<_{\varepsilon_X}t'$ by simultaneous recursion on $L_{\varepsilon_X}(s)$ resp.~$L_{\varepsilon_X}(t)+L_{\varepsilon_X}(t')$. It is standard to show the following:

\begin{lemma}[$\rca$]
 If $(X,<_X)$ is a linear order, then so is $(\varepsilon_X,<_{\varepsilon_X})$.
\end{lemma}

The crucial point is, once again, the preservation of well-foundedness:

\begin{proposition}[$\rca$]\label{prop:bh-yields-epsilon}
 The abstract Bachmann-Howard principle implies that $\varepsilon_X$ is well-founded for any well-order $X$.
\end{proposition}
\begin{proof}
 For a fixed well-order $X$ we consider the dilator $(T,\supp^T)$ with
 \begin{align*}
  T_Y&=(\{\bot \}\cup X)\times\omega^Y,\\[1ex]
  (x,\sigma)<_{T_Y}(x',\sigma')&\Leftrightarrow\begin{cases}
                                    \text{either $x<_Xx'$ (with $\bot<_Xx'$ for any $x'\in X$),}\\[1ex]
                                    \text{or $x=x'$ and $\sigma<_{\omega^Y}\sigma'$,}
                                   \end{cases}\\[1ex]
  T_f(x,\langle y_0,\dots,y_{n-1}\rangle)&=(x,\langle f(y_0),\dots,f(y_{n-1})\rangle)\quad\text{(with $f:Y\rightarrow Y'$),}\\[1ex]
  \supp^T_Y(x,\langle y_0,\dots,y_{n-1}\rangle)&=\{y_0,\dots,y_{n-1}\}.
 \end{align*}
 As in the proof of Proposition~\ref{prop:exp-well-founded} one must officially work with the representations from Definition~\ref{def:coded-prae-dilator}: It is straightforward to see that they exist as sets and that they represent a coded prae-dilator. By Proposition~\ref{prop:exp-well-founded} the abstract Bachmann-Howard principle ensures that~$\omega^Y$ is well-founded for any well-order $Y$. We conclude that~$T_Y$ is well-founded, so that $(T,\supp^T)$ is indeed a dilator. Another application of the abstract Bachmann-Howard principle yields a well-order $Y$ with a Bachmann-Howard collapse
 \begin{equation*}
  \vartheta:T_Y\rightarrow Y.
 \end{equation*}
 To each term $t\in\varepsilon_X$ we associate its ``$\varepsilon$-degree'' $t^*\in \{\bot \}\cup X$ by the recursion
 \begin{equation*}
  0^*=\bot ,\qquad\varepsilon_x^*=x,\qquad (\omega^{t_0}+\dots+\omega^{t_n})^*=t_0^*.
 \end{equation*}
 In order to conclude we show that the function $f:\varepsilon_X\rightarrow Y$ with
 \begin{align*}
  f(0)&:=\vartheta(\bot ,\langle\rangle),\\
  f(\varepsilon_x)&:=\vartheta(x,\langle\rangle),\\
  f(\omega^{t_0}+\dots+\omega^{t_n})&:=\begin{cases}
                                      \vartheta(t_0^*,\langle f(t_0),\dots,f(t_n)\rangle) & \text{if $f(t_n)\leq_Y\dots\leq_Y f(t_0)$},\\
                                      \vartheta(\bot ,\langle\rangle) & \text{otherwise}
                                     \end{cases}
 \end{align*}
 is order preserving (thus the alternative in the case distinction will never apply). Let us argue by induction on $L_{\varepsilon_X}(s)+L_{\varepsilon_X}(t)$ to establish the implication
 \begin{equation*}
  s<_{\varepsilon_X}t\quad\Rightarrow\quad f(s)<_Y f(t).
 \end{equation*}
 As a first interesting case we assume that $s=\omega^{s_0}+\dots+\omega^{s_n}<_{\varepsilon_X}\varepsilon_x=t$ holds because of $s_0<_{\varepsilon_X}t$. By an auxiliary induction on $s_0$ we see $s_0^*<_X x=t^*$ (note that we could only infer $s_0^*\leq_Xt^*$ if $t$ was not of the form $\varepsilon_x$). The induction hypothesis yields $f(s_n)\leq_Y\dots\leq_Y f(s_0)<_Y f(t)=\vartheta(x,\langle\rangle)$. Thus we have
 \begin{align*}
  (s_0^*,\langle f(s_0),\dots,f(s_n)\rangle)&<_{T_Y}(x,\langle\rangle),\\
  \supp^T_Y(s_0^*,\langle f(s_0),\dots,f(s_n)\rangle)&=\{f(s_0),\dots,f(s_n)\}\lef_Y\vartheta(x,\langle\rangle).
 \end{align*}
 By the definition of Bachmann-Howard collapse we get
 \begin{equation*}
  f(s)=\vartheta(s_0^*,\langle f(s_0),\dots,f(s_n)\rangle)<_Y \vartheta(x,\langle\rangle)=f(t),
 \end{equation*}
 as desired. Let us also consider $s=\omega^{s_0}+\dots+\omega^{s_n}<_{\varepsilon_X}\omega^{t_0}+\dots+\omega^{t_m}=t$. The induction hypothesis yields $f(s_n)\leq_Y\dots\leq_Y f(s_0)$ and $f(t_m)\leq_Y\dots\leq_Y f(t_0)$, as well as $\langle f(s_0),\dots,f(s_n)\rangle<_{\omega^Y}\langle f(t_0),\dots,f(t_m)\rangle$. In view of $s_0\leq_{\varepsilon_X}t_0$ we also have $s_0^*\leq_X t_0^*$ and thus
 \begin{equation*}
  (s_0^*,\langle f(s_0),\dots,f(s_n)\rangle)<_{T_Y} (t_0^*,\langle f(t_0),\dots,f(t_m)\rangle).
 \end{equation*}
 Furthermore we can observe
 \begin{multline*}
  \supp^T_Y(s_0^*,\langle f(s_0),\dots,f(s_n)\rangle)\leq_Y^{\operatorname{fin}}\\
  \leq_Y^{\operatorname{fin}}\supp^T_Y(t_0^*,\langle f(t_0),\dots,f(t_m)\rangle)\lef_Y\vartheta(t_0^*,\langle f(t_0),\dots,f(t_m)\rangle),
 \end{multline*}
 where the second inequality relies on the definition of Bachmann-Howard collapse. Also by the latter we finally get
 \begin{equation*}
  f(s)=\vartheta(s_0^*,\langle f(s_0),\dots,f(s_n)\rangle)<_Y\vartheta(t_0^*,\langle f(t_0),\dots,f(t_m)\rangle)=f(t).
 \end{equation*}
 The remaining cases are straightforward. Thus $f:\varepsilon_X\rightarrow Y$ is an order embedding and $\varepsilon_X$ is well-founded.
\end{proof}

As the final step of our bootstrapping process we use the abstract Bachmann-Howard principle to show that a certain notation system $\vartheta_X$ is well-founded for any well-order $X$. The latter implies that every set is contained in a (countable coded) $\omega$-model of bar induction, as shown by Rathjen and Valencia Vizca\'ino~\cite{rathjen-model-bi}. We thus reach a statement which is stronger than the base theory $\atr$ of Theorem~\ref{thm:main-abstract}. The following coincides with \cite[Definition~2.6]{rathjen-model-bi}:

\begin{definition}
 For each linear order $(X,<_X)$ the set $\vartheta_X$ of terms, the relation ${<_{\vartheta_X}}\subseteq\vartheta_X\times\vartheta_X$ and the function $\cdot^*:\vartheta_X\rightarrow\vartheta_X$ are defined by the following simultaneous recursion:
 \begin{enumerate}[label=(\roman*)]
  \item The term $0$ is an element of $\vartheta_X$.
  \item The term $\Omega$ is an element of $\vartheta_X$.
  \item For each $x\in X$ we have a term $\mathfrak E_x\in\vartheta_X$.
  \item If $s$ is a term in $\vartheta_X$, then so ist $\vartheta s$.
  \item If $s_0,\dots,s_n$ are terms in $\vartheta_X$, then so is $\omega^{s_0}+\dots+\omega^{s_n}$, provided that
  \begin{itemize}
   \item either $n=0$ and $s_0$ is not of the form $\Omega,\mathfrak E_x$ or $\vartheta s'$,
   \item or $n>0$ and $s_n\leq_{\vartheta_X}\dots\leq_{\vartheta_X}s_0$.
  \end{itemize}
 \end{enumerate}
 The map $s\mapsto s^*$ is given by
 \begin{equation*}
  0^*=0,\quad \Omega^*=0,\quad \mathfrak E_x^*=0,\quad (\vartheta s)^*=\vartheta s,\quad (\omega^{s_0}+\dots+\omega^{s_n})^*=\max_{i\leq n}s_i^*,
 \end{equation*}
 where the maximum is taken with respect to $<_{\vartheta_X}$. For $s,t\in\vartheta_X$ we have $s<_{\vartheta_X}t$ if and only if one of the following holds:
 \begin{enumerate}[label=(\roman*')]
  \item We have $s=0$ and $t\neq 0$.
  \item We have $s=\Omega$ and
  \begin{itemize}
   \item either $t$ is of the form $\mathfrak E_x$,
   \item or $t=\omega^{t_0}+\cdots+\omega^{t_n}$ with $s\leq_{\vartheta_X}t_0$.
  \end{itemize}
  \item We have $s=\mathfrak E_x$ and
  \begin{itemize}
   \item either $t=\mathfrak E_y$ with $x<_X y$,
   \item or $t=\omega^{t_0}+\cdots+\omega^{t_n}$ with $s\leq_{\vartheta_X}t_0$.
  \end{itemize}
  \item We have $s=\vartheta s'$ and
  \begin{itemize}
   \item either $t=\vartheta t'$ with $s'<_{\vartheta_X} t'$ and $(s')^*<_{\vartheta_X}\vartheta t'$,
   \item or $t=\vartheta t'$ with $\vartheta s'\leq_{\vartheta_X}(t')^*$,
   \item or $t$ is of the form $\Omega$ or $\mathfrak E_x$,
   \item or $t=\omega^{t_0}+\cdots+\omega^{t_n}$ with $s\leq_{\vartheta_X}t_0$.
  \end{itemize}
  \item We have $s=\omega^{s_0}+\dots+\omega^{s_n}$ and
  \begin{itemize}
   \item either $t$ is of the form $\Omega,\mathfrak E_x$ or $\vartheta t'$ and $s_0<_{\vartheta_X} t$,
   \item or $t=\omega^{t_0}+\dots+\omega^{t_m}$ and one of the following holds:
   \begin{itemize}
    \item Either we have $n<m$ and $s_i=t_i$ for all $i\leq n$,
    \item or there is a $j\leq\min\{n,m\}$ with $s_j<_{\vartheta_X} t_j$ and $s_i=t_i$ for $i<j$.
   \end{itemize}
  \end{itemize}
 \end{enumerate}
\end{definition}

To formalize this in $\rca$ one starts with a term system $\vartheta^0_X\supseteq\vartheta_X$ that ignores the condition $s_n\leq_{\vartheta_X}\dots\leq_{\vartheta_X}s_0$ in clause (v). One then defines $L_{\vartheta_X}:\vartheta^0_X\rightarrow\mathbb N$ by
\begin{equation*}
 L_{\vartheta_X}(t):=\begin{cases}
                        \max\{\ulcorner 0\urcorner,\ulcorner t\urcorner\} & \text{if $t=0,t=\Omega$ or $t=\mathfrak E_x$},\\
                        \max\{\ulcorner t\urcorner,L_{\vartheta_X}(t')+1\} & \text{if $t=\vartheta t'$},\\
                        \max\{\ulcorner t\urcorner,L_{\vartheta_X}(t_0)+\dots+L_{\vartheta_X}(t_n)+1\} & \text{if $t=\omega^{t_0}+\dots+\omega^{t_n}$}.
                       \end{cases}
\end{equation*}
The occurrence of $\ulcorner 0\urcorner$ in the first case ensures $L_{\vartheta_X}(s^*)\leq L_{\vartheta_X}(s)$. To decide $s\in\vartheta_X$ and $t<_{\vartheta_X}t'$ and to compute $s^*$ one proceeds by simultaneous recursion on $L_{\vartheta_X}(s)$ resp.~$L_{\vartheta_X}(t)+L_{\vartheta_X}(t')$. By \cite[Lemma~2.7]{rathjen-model-bi} we have the following:

\begin{lemma}[$\rca$]
 If $(X,<_X)$ is a linear order, then so is $(\vartheta_X,<_{\vartheta_X})$.
\end{lemma}

Now we come to the main technical result of the present section:

\begin{theorem}[$\rca$]
 The abstract Bachmann-Howard principle implies that $\vartheta_X$ is well-founded for any well-order $X$.
\end{theorem}
\begin{proof}
 For a fixed well-order $X$ and an arbitrary linear order $Y$ we put
 \begin{equation*}
  T_Y:=\varepsilon_{Y\cup\{\Omega\}\cup X}.
 \end{equation*}
 Here $Y\cup\{\Omega\}\cup X$ is ordered as written: Any element of $Y$ is smaller than the constant~$\Omega$, which is in turn smaller than any element of $X$. By Definition~\ref{def:vareps-X} we obtain an order $<_{T_Y}$ on $T_Y$. For each embedding $f:Y\rightarrow Y'$ we define an embedding $T_f:T_Y\rightarrow T_{Y'}$ by the recursion
 \begin{equation*}
  T_f(t)=\begin{cases}
          t & \text{if $t=0,t=\varepsilon_\Omega$ or $t=\varepsilon_x$ with $x\in X$},\\
          \varepsilon_{f(y)} & \text{if $t=\varepsilon_y$ with $y\in Y$},\\
          \omega^{T_f(t_0)}+\dots+\omega^{T_f(t_n)} & \text{if $t=\omega^{t_0}+\dots+\omega^{t_n}$.}
         \end{cases}
 \end{equation*}
 It is straightforward to verify that $Y\mapsto T_Y$ and $f\mapsto T_f$ form an endofunctor of linear orders. To obtain a prae-dilator we define $\supp^T_Y:T_Y\rightarrow[Y]^{<\omega}$ by
 \begin{equation*}
  \supp^T_Y(t)=\begin{cases}
          \emptyset & \text{if $t=0,t=\varepsilon_\Omega$ or $t=\varepsilon_x$ with $x\in X$},\\
          \{y\} & \text{if $t=\varepsilon_y$ with $y\in Y$},\\
          \textstyle\bigcup_{i\leq n}\supp^T_Y(t_i) & \text{if $t=\omega^{t_0}+\dots+\omega^{t_n}$.}
         \end{cases}
 \end{equation*}
 By Proposition~\ref{prop:bh-yields-epsilon} the abstract Bachmann-Howard principle ensures that $\varepsilon_{Y\cup\{\Omega\}\cup X}$ is well-founded for any well-order $Y$. Thus $T$ is a dilator. Another application of the abstract Bachmann-Howard principle yields a well-order $Y$ with a Bachmann-Howard collapse
 \begin{equation*}
  \vartheta:T_Y\rightarrow Y.
 \end{equation*}
 In order to conclude we show that the function $f:\vartheta_X\rightarrow T_Y=\varepsilon_{Y\cup\{\Omega\}\cup X}$ with
 \begin{align*}
  f(0)&=0,\\
  f(\Omega)&=\varepsilon_{\Omega},\\
  f(\mathfrak E_x)&=\varepsilon_x,\\
  f(\vartheta t)&=\varepsilon_{\vartheta(f(t))},\\
  f(\omega^{t_0}+\dots+\omega^{t_n})&=\omega^{f(t_0)}+\dots+\omega^{f(t_n)}
 \end{align*}
 is order preserving. To establish the implication
 \begin{equation*}
  s<_{\vartheta_X}t\quad\Rightarrow\quad f(s)<_{T_Y} f(t)
 \end{equation*}
 we argue by induction on $L_{\vartheta_X}(s)+L_{\vartheta_X}(t)$. Simultaneously one must verify
 \begin{equation*}
  f(r^*)=\textstyle\max_{<_{T_Y}}(\{0\}\cup\{\varepsilon_y\,|\,y\in\supp^T_Y(f(r))\})
 \end{equation*}
 by induction on $L_{\vartheta_X}(r)$. The only interesting case is
 \begin{equation*}
  s=\vartheta s'<_{\vartheta_X}\vartheta t'=t.
 \end{equation*}
 We have to consider two possibilities: First assume $s'<_{\vartheta_X}t'$ and $(s')^*<_{\vartheta_X}\vartheta t'$. By the simultaneous induction hypothesis we get
 \begin{equation*}
  \{\varepsilon_y\,|\,y\in\supp^T_Y(f(s'))\}\leq^{\text{fin}}_{T_Y}f((s')^*)<_{T_Y}f(\vartheta t')=\varepsilon_{\vartheta(f(t'))},
 \end{equation*}
 which implies $\supp^T_Y(f(s'))\lef_Y\vartheta(f(t'))$. The induction hypothesis also provides the inequality $f(s')<_{T_Y}f(t')$. By the definition of Bachmann-Howard collapse we obtain $\vartheta(f(s'))<_Y\vartheta(f(t'))$ and then
 \begin{equation*}
  f(s)=\varepsilon_{\vartheta(f(s'))}<_{T_Y}\varepsilon_{\vartheta(f(t'))}=f(t).
 \end{equation*}
 Now assume that $s=\vartheta s'<_{\vartheta_X}\vartheta t'=t$ holds because of $s\leq_{\vartheta_X}(t')^*$. By the definition of Bachmann-Howard collapse we have $\supp^T_Y(f(t'))\lef_Y\vartheta(f(t'))$. Together with the simultaneous induction hypothesis and $0<_{T_Y}\varepsilon_{\vartheta(f(t'))}$ we can infer
 \begin{equation*}
  f(s)\leq_{T_Y}f((t')^*)=\textstyle\max_{<_{T_Y}}(\{0\}\cup\{\varepsilon_y\,|\,y\in\supp^T_Y(f(t'))\})<_{T_Y}\varepsilon_{\vartheta(f(t'))}=f(t),
 \end{equation*}
 as desired.
\end{proof}

Note that the case $X=\emptyset$ of the previous theorem yields the well-foundedness of the usual Bachmann-Howard ordinal. The following result completes our bootstrapping: It allows us to lower the base theory to $\rca$ (cf.~Theorem~\ref{thm:main-text} below).

\begin{corollary}[$\rca$]\label{cor:bhp-atr}
 The abstract Bachmann-Howard principle implies all axioms of $\atr$.
\end{corollary}
\begin{proof}
 According to Proposition~\ref{prop:exp-well-founded} the abstract Bachmann-Howard principle implies that $\omega^X$ is well-founded for any well-order $X$. By the aforementioned result of Girard~\cite[Section~II.5]{girard87} and Hirst~\cite{hirst94} this secures arithmetical comprehension. The latter allows us to argue in terms of countable coded $\omega$-models and valuations of formulas in these models. The previous theorem tells us that $\vartheta_X$ is well-founded for any well-order~$X$. As shown by Rathjen and Valencia Vizca\'ino~\cite{rathjen-model-bi} this implies that any set is countained in a countable coded $\omega$-model of bar induction. To conclude we recall that the axioms of $\atr$ are provable by bar induction and have complexity~$\Pi^1_2$ (see~\cite[Corollary~VII.2.19]{simpson09}).
\end{proof}

\section{Computing a Bachmann-Howard Fixed Point}

In this section we construct a notation system $\vartheta(T)$ for the smallest Bachmann-Howard fixed point of a coded prae-dilator $T$. The point is that $\vartheta(T)$ is computable relative to $T$, so that $\rca$ proves its existence as a set and indeed a linear order. The statement that $\vartheta(T)$ is well-founded for any coded dilator $T$ will be called the computable Bachmann-Howard principle. We will show that it is equivalent to the abstract Bachmann-Howard principle and thus to $\Pi^1_1$-comprehension.

To understand the construction of $\vartheta(T)$, assume that we have a Bachmann-Howard collapse $\vartheta:D^T_{\vartheta(T)}\rightarrow\vartheta(T)$. Any element of $D^T_{\vartheta(T)}$ is of the form $\langle a,\sigma\rangle$, where $a$ is a finite subset of $\vartheta(T)$ and $\sigma\in T_{|a|}$ satisfies $\supp^T_{|a|}(\sigma)=|a|$. Write $a=\{s_0,\dots,s_{n-1}\}$ with $s_0<_{\vartheta(T)}\dots<_{\vartheta(T)}s_{n-1}$. The idea is to represent the collapsed element $\vartheta(\langle a,\sigma\rangle)\in\vartheta(T)$ by the term $\vartheta_\sigma^{s_0,\dots,s_{n-1}}$.

\begin{definition}[$\rca$]\label{def:computable-bh-fixed-point}
For each coded prae-dilator $T$ the set $\vartheta(T)$ and the relation ${<_{\vartheta(T)}}\subseteq\vartheta(T)\times\vartheta(T)$ are defined by the following simultaneous recursion:
\begin{enumerate}[label=(\roman*)]
\item If we have elements $s_0<_{\vartheta(T)}\dots<_{\vartheta(T)}s_{n-1}$ of $\vartheta(T)$ and an element $\sigma\in T_n$ with $\supp^T_n(\sigma)=n$, then the term $\vartheta_\sigma^{s_0,\dots,s_{n-1}}$ is an element of $\vartheta(T)$ as well.
\end{enumerate}
Given elements $s=\vartheta_\sigma^{s_0,\dots,s_{n-1}}$ and $t=\vartheta_\tau^{t_0,\dots,t_{m-1}}$ of $\vartheta(T)$, we have $s<_{\vartheta(T)} t$ precisely if one of the following holds:
\begin{enumerate}[label=(\roman*')]
\item We have $T_f(\sigma)<_{T_k}T_g(\tau)$ for some strictly increasing functions
\begin{equation*}
f:n\rightarrow k:=|\{s_0,\dots,s_{n-1},t_0,\dots,t_{m-1}\}|\quad\text{and}\quad g:m\rightarrow k
\end{equation*}
with $f(i)<g(j)\Leftrightarrow s_i<_{\vartheta(T)} t_j$ and $g(j)<f(i)\Leftrightarrow t_j<_{\vartheta(T)} s_i$. Furthermore we have $s_{n-1}<_{\vartheta(T)}t$ or $n=0$.
\item We have $m>0$ and $s\leq_{\vartheta(T)} t_{m-1}$.
\end{enumerate}
\end{definition}

Note that $n=0$ is permitted in clause (i), leading to initial terms $\vartheta_\sigma^{\langle\rangle}$ with empty upper index. Thus the set $\vartheta(T)$ is empty if and only if $T_0$ is. The formulation of clause (i') is somewhat awkward because we do not yet know that $<_{\vartheta(T)}$ is a linear order. Once this fact is established we see that $f$ and $g$ are the unique functions that make the following diagram commute, where the vertical arrows are the increasing enumerations with respect to $<_{\vartheta(T)}$:
\begin{equation*}
\begin{tikzcd}
n\arrow[r]\arrow[d,"f"'] & \{s_0,\dots,s_{n-1}\}\arrow[d,hook']\\
k\arrow[r] & \{s_0,\dots,s_{n-1},t_0,\dots,t_{m-1}\}\\
m\arrow[u,"g"]\arrow[r] & \{t_0,\dots,t_{m-1}\}\arrow[u,hook]
\end{tikzcd}
\end{equation*}
To formalize Definition~\ref{def:computable-bh-fixed-point} in $\rca$ one starts with a set $\vartheta^0(T)\supseteq\vartheta(T)$ of terms that ignores the condition $s_0<_{\vartheta(T)}\dots<_{\vartheta(T)}s_{n-1}$ in clause~(i). Then consider the length function~$L_{\vartheta(T)}:\vartheta^0(T)\rightarrow\mathbb N$ defined by
\begin{equation*}
L_{\vartheta(T)}(s)=\begin{cases}
                     \ulcorner s\urcorner & \text{if $s=\vartheta_\sigma^{\langle\rangle}$},\\
                     \max\{\ulcorner s\urcorner,2\cdot L_{\vartheta(T)}(s_0)+\dots+2\cdot L_{\vartheta(T)}(s_n)+1\} & \text{if $s=\vartheta_\sigma^{s_0,\dots,s_n}$}.
                    \end{cases}
\end{equation*}
Now one can decide $r\in\vartheta(T)$ and $s<_{\vartheta(T)} t$ by simultaneous induction on $L_{\vartheta(T)}(r)$ resp.~$L_{\vartheta(T)}(s)+L_{\vartheta(T)}(t)$. As in the previous section we have included the G\"odel number $\ulcorner s\urcorner$ in order to ensure that $\forall_s(L_{\vartheta(T)}(s)\leq n\rightarrow\cdots)$ amounts to a bounded quantifier. The significance of the factor $2$ becomes clear in the following proof:

\begin{proposition}[$\rca$]\label{prop:notation-bh-fixed-linear}
For any coded prae-dilator $T$ the relation $<_{\vartheta(T)}$ is a linear order on $\vartheta(T)$.
\end{proposition}
\begin{proof}
By simultaneous induction on $n$ one shows that
\begin{gather*}
2\cdot L_{\vartheta(T)}(s)\leq n\rightarrow s\not<_{\vartheta(T)} s,\\
L_{\vartheta(T)}(s)+L_{\vartheta(T)}(t)\leq n\rightarrow s<_{\vartheta(T)} t\lor s=t\lor t<_{\vartheta(T)} s,\\
L_{\vartheta(T)}(r)+L_{\vartheta(T)}(s)+L_{\vartheta(T)}(t)\leq n\rightarrow(r<_{\vartheta(T)} s\land s<_{\vartheta(T)} t\rightarrow r<_{\vartheta(T)} t)
\end{gather*}
holds for all $r,s,t\in\vartheta(T)$. To establish antisymmetry we write $s=\vartheta_\sigma^{s_0,\dots,s_{n-1}}$. Aiming at a contradiction, assume first that $s<_{\vartheta(T)} s$ holds by clause~(i') of Definition~\ref{def:computable-bh-fixed-point}. The functions $f$ and $g$ can only be the identity on $n=k=m$. Thus we would have to have $\sigma=T_f(\sigma)<_{T_n}T_g(\sigma)=\sigma$, contradicting the antisymmetry of~$<_{T_n}$. Now assume that $s<_{\vartheta(T)} s$ holds by clause~(ii'), which means that we have $n>0$ and~$s\leq_{\vartheta(T)}s_{n-1}$. On the other hand we have $s_{n-1}<_{\vartheta(T)}s$, by clause (ii') and the trivial inequality $s_{n-1}\leq_{\vartheta(T)}s_{n-1}$. In view of
\begin{equation*}
 L_{\vartheta(T)}(s_{n-1})+L_{\vartheta(T)}(s)+L_{\vartheta(T)}(s_{n-1})<2\cdot L_{\vartheta(T)}(s)
\end{equation*}
we can invoke the induction hypothesis to get $s_{n-1}<_{\vartheta(T)}s_{n-1}$ by transitivity. This contradicts the antisymmetry available by induction hypothesis.

To establish trichotomy we write $s=\vartheta^{s_0,\dots,s_{n-1}}_\sigma$ and $t=\vartheta^{t_0,\dots,t_{m-1}}_\tau$. The induction hypothesis implies that $<_{\vartheta(T)}$ is linear on $\{s_0,\dots,s_{n-1},t_0,\dots,t_{m-1}\}$ (in particular the induction hypothesis covers $s_i<_{\vartheta(T)}t_j\land t_j<_{\vartheta(T)}s_i\rightarrow s_i<_{\vartheta(T)}s_i$, due to the factor $2$ in the definition of $L_{\vartheta(T)}$). Thus we can consider the unique functions $f$ and $g$ that make the above diagram commute. First assume $T_f(\sigma)=T_g(\tau)$. Since $\supp^T$ is a natural transformation we get
\begin{multline*}
[f]^{<\omega}(n)=[f]^{<\omega}(\supp^T_n(\sigma))=\supp^T_k(T_f(\sigma))=\\
=\supp^T_k(T_g(\tau))=[g]^{<\omega}(\supp^T_m(\tau))=[g]^{<\omega}(m).
\end{multline*}
Together with $[f]^{<\omega}(n)\cup [g]^{<\omega}(m)=k$ this implies that $f$ and $g$ must be the identity on $n=k=m$. Thus we obtain $\sigma=\tau$ and $\langle s_0,\dots,s_{n-1}\rangle=\langle t_0,\dots,t_{m-1}\rangle$, which means $s=t$. Now let us assume $T_f(\sigma)<_{T_k}T_g(\tau)$. If we have $n=0$, then we get $s<_{\vartheta(T)} t$ by clause~(i'). If we have $n>0$, then the induction hypothesis yields $s_{n-1}<_{\vartheta(T)} t$ or $t\leq_{\vartheta(T)} s_{n-1}$. In the first case we get $s<_{\vartheta(T)} t$ by clause~(i'), while the second case yields $t<_{\vartheta(T)} s$ by clause~(ii'). For $T_g(\tau)<_{T_k}T_f(\sigma)$ the argument is symmetric.

Finally, we establish transitivity: Consider terms $r=\vartheta_\rho^{r_0,\dots,r_{l-1}}$, $s=\vartheta_\sigma^{s_0,\dots,s_{n-1}}$ and $t=\vartheta_\tau^{t_0,\dots,t_{m-1}}$ with $r<_{\vartheta(T)}s$ and $s<_{\vartheta(T)}t$. First assume that $s<_{\vartheta(T)}t$ holds by clause (ii'), i.e.~that we have $m>0$ and $s\leq_{\vartheta(T)}t_{m-1}$. Then the induction hypothesis yields $r<_{\vartheta(T)}t_{m-1}$, so that we get $r<_{\vartheta(T)}t$ by clause~(ii'). Now assume that $s<_{\vartheta(T)}t$ holds by clause~(i') while $r<_{\vartheta(T)}s$ holds by clause~(ii'). Then we have $n>0$ and $r\leq_{\vartheta(T)}s_{n-1}<_{\vartheta(T)}t$, so that the induction hypothesis yields $r<_{\vartheta(T)} t$. Finally, assume that both inequalities hold by clause~(i'). Then we have $l=0$ or $r_{l-1}<_{\vartheta(T)}s$, which yields $r_{l-1}<_{\vartheta(T)}t$ by induction hypothesis. To see that the remaining condition of clause~(i') is transitive one completes the above diagram by the inclusions into $\{r_0,\dots,r_{l-1},s_0,\dots,s_{n-1},t_0,\dots,t_{m-1}\}$.
\end{proof}

We can now show a central result of this paper: The theory $\rca$ proves the existence of Bachmann-Howard fixed points (but in general it will not prove their well-foundedness).

\begin{theorem}[$\rca$]\label{thm:computable-is-bh-fixed-point}
 Given any coded prae-dilator $T$, the linear order $\vartheta(T)$ is a Bachmann-Howard fixed point of $T$.
\end{theorem}
\begin{proof}
 We must construct a Bachmann-Howard collapse $\vartheta:D^T_{\vartheta(T)}\rightarrow\vartheta(T)$. In view of Definition~\ref{def:reconstruct-dilators} we set
 \begin{equation*}
  \vartheta(\langle a,\sigma\rangle)=\vartheta_\sigma^{s_0,\dots,s_{n-1}}\quad\text{for $a=\{s_0,\dots,s_{n-1}\}$ with $s_0<_{\vartheta(T)}\dots<_{\vartheta(T)}s_{n-1}$}.
 \end{equation*}
 Let us verify the conditions from Definition~\ref{def:bachmann-howard-collapse}: Aiming at condition~(i) we assume $\langle a,\sigma\rangle<_{D^T_{\vartheta(T)}}\langle b,\tau\rangle$. By Definition~\ref{def:reconstruct-dilators} this means $T_{|\iota_a^{a\cup b}|}(\sigma)<_{T_{|a\cup b|}}T_{|\iota_b^{a\cup b}|}(\tau)$, where $\iota_a^{a\cup b}$ and $\iota_b^{a\cup b}$ are the inclusion maps from $a$ resp.~$b$ into $a\cup b$. Write $a=\{s_0,\dots,s_{n-1}\}$ and $b=\{t_0,\dots,t_{m-1}\}$ in increasing order, and observe that $|\iota_a^{a\cup b}|$ and $|\iota_b^{a\cup b}|$ coincide with the functions $f$ and $g$ from Definition~\ref{def:computable-bh-fixed-point}(i'). Assuming the side condition of Definition~\ref{def:bachmann-howard-collapse}(i) we also get
 \begin{equation*}
  \{s_0,\dots,s_{n-1}\}=\supp^{D^T}_{\vartheta(T)}(\langle a,\sigma\rangle)\lef_{\vartheta(T)}\vartheta(\langle b,\tau\rangle)=\vartheta_\tau^{t_0,\dots,t_{m-1}}.
 \end{equation*}
 Thus we must have $s_{n-1}<_{\vartheta(T)}\vartheta_\tau^{t_0,\dots,t_{m-1}}$ or $n=0$. Now Definition~\ref{def:computable-bh-fixed-point}(i') yields
 \begin{equation*}
  \vartheta(\langle a,\sigma\rangle)=\vartheta_\sigma^{s_0,\dots,s_{n-1}}<_{\vartheta(T)}\vartheta_\tau^{t_0,\dots,t_{m-1}}=\vartheta(\langle b,\tau\rangle),
 \end{equation*}
 as required by condition (i) of Definition~\ref{def:bachmann-howard-collapse}. To establish condition~(ii) we consider $\langle a,\sigma\rangle\in D^T_{\vartheta(T)}$ and write $a=\{s_0,\dots,s_{n-1}\}$ in increasing order. In case $n>0$ we observe
 \begin{equation*}
  s_0<_{\vartheta(T)}\dots<_{\vartheta(T)}s_{n-1}<_{\vartheta(T)}\vartheta_\sigma^{s_0,\dots,s_{n-1}}=\vartheta(\langle a,\sigma\rangle),
 \end{equation*}
 where the last inequality holds by clause (ii') of Definition~\ref{def:computable-bh-fixed-point}. This implies
 \begin{equation*}
  \supp^{D^T}_{\vartheta(T)}(\langle a,\sigma\rangle)=\{s_0,\dots,s_{n-1}\}\lef_{\vartheta(T)}\vartheta(\langle a,\sigma\rangle),
 \end{equation*}
 just as condition (ii) of Definition~\ref{def:bachmann-howard-collapse} demands.
\end{proof}

In view of the theorem, the following assertion is at least as strong as (the second-order version of) the abstract Bachmann-Howard principle:

\begin{definition}[$\rca$]
 The computable Bachmann-Howard principle is the statement that $\vartheta(T)$ is well-founded for any coded dilator $T$.
\end{definition}

To see that the abstract Bachmann-Howard principle implies its computable counterpart we show that the Bachmann-Howard fixed point $\vartheta(T)$ is minimal:

\begin{theorem}[$\rca$]\label{thm:comp-fixed-point-smallest}
 Consider a coded prae-dilator $T$. The order $\vartheta(T)$ can be embedded into any Bachmann-Howard fixed point of $T$.
\end{theorem}
\begin{proof}
 Let $\vartheta:D^T_X\rightarrow X$ be a Bachmann-Howard collapse. The desired embedding $f:\vartheta(T)\rightarrow X$ can be recursively defined by
 \begin{equation*}
  f(\vartheta_\sigma^{s_0,\dots,s_{n-1}})=\vartheta(\langle\{f(s_0),\dots,f(s_{n-1})\},\sigma\rangle).
 \end{equation*}
 Inductively we may assume that $f$ is order preserving on $\{s_0,\dots,s_{n-1}\}$. In particular we have $|\{f(s_0),\dots,f(s_{n-1})\}|=n$, so that $\langle\{f(s_0),\dots,f(s_{n-1})\},\sigma\rangle$ is indeed an element of $D^T_X$. The implication
 \begin{equation*}
  s<_{\vartheta(T)}t\quad\Rightarrow\quad f(s)<_Xf(t)
 \end{equation*}
can be established by induction on $L_{\vartheta(T)}(s)+L_{\vartheta(T)}(t)$: Let us write $s=\vartheta_\sigma^{s_0,\dots,s_{n-1}}$ and $t=\vartheta_\tau^{t_0,\dots,t_{m-1}}$, as well as $a=\{s_0,\dots,s_{n-1}\}$ and $b=\{t_0,\dots,t_{m-1}\}$. First assume that $s<_{\vartheta(T)}t$ holds by clause~(i') of Definition~\ref{def:computable-bh-fixed-point}. This means that we have $T_{|\iota_a^{a\cup b}|}(\sigma)<_{T_{|a\cup b|}}T_{|\iota_b^{a\cup b}|}(\tau)$, using the notation from Section~\ref{sect:dilator-prs-so}. The induction hypothesis ensures that $f$ is order preserving on $a\cup b$, which implies
 \begin{equation*}
  \left|\iota_a^{a\cup b}\right|=\left|\iota_{[f]^{<\omega}(a)}^{[f]^{<\omega}(a)\cup [f]^{<\omega}(b)}\right|\quad\text{and}\quad\left|\iota_b^{a\cup b}\right|=\left|\iota_{[f]^{<\omega}(b)}^{[f]^{<\omega}(a)\cup [f]^{<\omega}(b)}\right|.
 \end{equation*}
 In view of Definition~\ref{def:reconstruct-dilators} we can infer
 \begin{equation*}
  \langle[f]^{<\omega}(a),\sigma\rangle<_{D^T_X}\langle[f]^{<\omega}(b),\tau\rangle.
 \end{equation*}
 In case $n>0$ we observe $s_0<_{\vartheta(T)}\dots<_{\vartheta(T)}s_{n-1}<_{\vartheta(T)}s<_{\vartheta(T)}t$, where the penultimate inequality holds by clause (ii') of Definition~\ref{def:computable-bh-fixed-point}. By induction hypothesis we get $f(s_0)<_X\dots<_Xf(s_{n-1})<_Xf(t)$ and thus
 \begin{equation*}
  \supp^{D^T}_X(\langle[f]^{<\omega}(a),\sigma\rangle)=[f]^{<\omega}(a)\lef_X f(t)=\vartheta(\langle[f]^{<\omega}(b),\tau\rangle).
 \end{equation*}
 Then condition~(i) of Definition~\ref{def:bachmann-howard-collapse} yields
 \begin{equation*}
  f(s)=\vartheta(\langle[f]^{<\omega}(a),\sigma\rangle)<_X\vartheta(\langle[f]^{<\omega}(b),\tau\rangle)=f(t),
 \end{equation*}
 as desired. Now assume that $s<_{\vartheta(T)}t$ holds by clause~(ii') of Definition~\ref{def:computable-bh-fixed-point}. This means that we have $m>0$ and $s\leq_{\vartheta(T)}t_{m-1}$. By induction hypothesis we can infer~$f(s)\leq_Xf(t_{m-1})$. Using clause~(ii) of Definition~\ref{def:bachmann-howard-collapse} we also get
 \begin{equation*}
  f(t_{m-1})\in [f]^{<\omega}(b)=\supp^{D^T}_X(\langle[f]^{<\omega}(b),\tau\rangle)\lef_X\vartheta(\langle[f]^{<\omega}(b),\tau\rangle)=f(t),
 \end{equation*}
 which yields $f(t_{m-1})<_X f(t)$. By transitivity we get $f(s)<_Xf(t)$.
\end{proof}

Putting results together, we obtain the following refinement of Theorem~\ref{thm:main-abstract}:

\begin{theorem}\label{thm:main-text}
 The following are equivalent over $\rca$:
 \begin{enumerate}[label=(\roman*)]
  \item The principle of $\Pi^1_1$-comprehension.
  \item The abstract Bachmann-Howard principle.
  \item The computable Bachmann-Howard principle.
 \end{enumerate}
\end{theorem}

We remark that statement~(ii) refers to the second-order version of the abstract Bachmann-Howard principle (cf.~Definition~\ref{def:so-bh}). Over $\atrs$ the latter is equivalent to the set-theoretic version of the abstract Bachmann-Howard principle (cf.~Definition~\ref{def:abstract-bhp}), as we have shown in Lemma~\ref{lem:abstract-bhp-set-so}.

\begin{proof}
 By Theorem~\ref{thm:main-abstract} (and the aforementioned Lemma~\ref{lem:abstract-bhp-set-so}) the equivalence between (i) and (ii) can be proved in Simpson's set-theoretic version of $\atr$. By conservativity (see~\cite{simpson82,simpson09}) it is provable in $\atr$ itself. Since $\Pi^1_1$-comprehension implies arithmetical transfinite recursion, the implication~(i)$\Rightarrow$(ii) is already provable in~$\rca$. In view of Corollary~\ref{cor:bhp-atr} the same holds for the implication~(ii)$\Rightarrow$(i). Aiming at (ii)$\Rightarrow$(iii), we invoke the abstract Bachmann-Howard principle to get a well-founded Bachmann-Howard fixed point~$X$ of a given coded dilator~$T$. By Theorem~\ref{thm:comp-fixed-point-smallest} there is an order embedding of $\vartheta(T)$ into~$X$. Thus $\vartheta(T)$ is well-founded as well, as required by the computable Bachmann-Howard principle. Finally, the implication~(iii)$\Rightarrow$(ii) follows from the fact that $\vartheta(T)$ is a Bachmann-Howard fixed point of $T$, as established in Theorem~\ref{thm:computable-is-bh-fixed-point}.
\end{proof}

The merit of Theorem~\ref{thm:main-text} is that it pinpoints the computational content: It shows that the strength of the Bachmann-Howard principle lies uniquely in the preservation of well-foundedness, not in the existence of a linearly ordered Bachmann-Howard fixed point as such. It would be interesting to use methods from computability theory (similar to those in~\cite{marcone-montalban}) to analyze the computable Bachmann-Howard principle, or indeed its contrapositive: Can one describe a computable prae-dilator~$T$ with a computable descending sequence in $\vartheta(T)$, such that the hyper\-jump is computable from any witness to the fact that $T$ fails to be a dilator?

\bibliographystyle{amsplain}
\bibliography{Bibliography_Freund}

\end{document}